\documentclass{article}

\usepackage{amsmath}
\usepackage{amsthm}
\usepackage{mathtools}
\usepackage{amsfonts}
\usepackage{mathtools}
\usepackage{amssymb,xcolor}
\usepackage{amsfonts}
\usepackage{graphicx}
\usepackage[draft,ulem=normalem]{changes}
\usepackage{xspace}
\usepackage{url}

\newtheorem {theorem}{Theorem}[section]
\newtheorem{assumption}[theorem]{Assumption}
\newtheorem{remark}[theorem]{{\itshape Remark}}
\newtheorem{lemma}[theorem]{Lemma}
\newtheorem{proposition}[theorem]{Proposition}

\newcommand{\Pb}{\mbox{\rm (P)}\xspace}
\newcommand{\Uad}{\mathcal{U}_{ad}}
\newcommand{\dx}{\,\mathrm{d}x}
\newcommand{\dive}{\text{div}}

\newcommand{\Ur}{\mathcal{U}_\rho}
\newcommand{\Pbe}{\mbox{\rm (P$_\varepsilon$)}\xspace}
\newcommand{\Pbek}{\mbox{\rm (P$_{\varepsilon_k}$)}\xspace}
\newcommand{\Pbr}{\mbox{\rm (P$_\rho$)}\xspace}
\newcommand{\Pbre}{\mbox{\rm (P$_{\rho,\varepsilon}$)}\xspace}

\numberwithin{equation}{section}

\title{New assumptions for stability analysis in elliptic optimal control problems\thanks{The first author was supported by MCIN/ AEI/10.13039/501100011033/ under research project PID2020-114837GB-I00. The second and third authors were supported by the Austrian Science Foundation (FWF) under grant No I4571.}}

\author{Eduardo Casas\thanks{Departamento de Matem\'atica Aplicada y Ciencias de la Computaci\'on, E.T.S.I. Industriales y de Telecomunicaci\'on, Universidad de Cantabria, Santander, Espa\~na, {\tt eduardo.casas@unican.es}}
\and Alberto Dom\'inguez Corella\thanks{Institute of Statistics and Mathematical Methods in Economics,
		Vienna University of Technology, Austria, {\tt alberto.corella@tuwien.ac.at}}
\and Nicolai Jork\thanks{Institute of Statistics and Mathematical Methods in Economics,
	Vienna University of Technology, Austria, {\tt nicolai.jork@tuwien.ac.at}}}

\begin{document}

\maketitle

\begin{abstract}
This paper is dedicated to the stability analysis of the optimal solutions of a control problem associated with a semilinear elliptic equation. The linear differential operator of the equation is neither monotone nor coercive due to the presence of a convection term. The control appears only linearly, or even it can not appear in an explicit form in the objective functional. Under new assumptions, we prove Lipschitz stability of the optimal controls and associated states with respect to perturbations in the equation and the objective functional as well as with respect to the Tikhonov regularization parameter.
\end{abstract}

\pagestyle{myheadings} \thispagestyle{plain} \markboth{E.~CASAS, A.~DOM\'{I}NGUEZ CORELLA, AND N.~JORK}{}

\section{Introduction}
\label{S1}

In this paper, we study the following optimal control problem
\[
   \Pb \  \min_{u \in \Uad} J(u) := \int_\Omega L(x,y_u(x),u(x))\dx,
\]
where $\Uad = \{u \in L^2(\Omega) : u_a \le u(x) \le u_b\ \text{ for a.a. } x \in \Omega\}$, $-\infty < u_a < u_b < +\infty$. Here, $y_u$ denotes the solution of the semilinear elliptic equation:
\begin{align}
\left\{\begin{array}{l}
-\dive\big(A(x)\nabla y\big)+b(x) \cdot \nabla y+f(x,y) = u\  \text{ in }\ \Omega,
\\ y=0 \ \text{ on }\ \Gamma.
\end{array} \right.
\label{E1.1}
\end{align}

Assumptions on the data of the control problem \Pb will be given below. The aim of this paper is to prove stability results for the local minimizers of \Pb with respect to perturbations in the data of the control problem. There are quite a few previous papers devoted to this issue \cite{Griesse-ZAA2006}, \cite{Hinze-Meyer-COAP2012}, \cite{KNCY-AMO2021}, \cite{Malanowski-Troltzsch-CC2000}, just to mention some of them. In all these cases, the second derivative of $L$ with respect to $u$ is bounded from below by a positive constant. This is the case where the Tikhonov term is involved in the objective functional. Under this condition and assuming sufficient second order optimality conditions (SSOC), the Lipschitz stability of the optimal controls is proved. Here, we assume that $u$ appears linearly in $L(x,y,u)$ or even it does not appear at all. Therefore, the previous results do not apply. In this case, under (SSOC) for optimality, Lipschitz stability of the optimal states can be proved; see \cite{CRT2015}. In Section \ref{S4}, we obtain analogous estimates for the optimal states replacing (SSOC) by a weaker condition; see \eqref{E3.13}. It is weaker because (SSOC) implies our assumption, but they are not equivalent. In addition, our assumption implies strict local optimality of the control; see Theorem \ref{T3.3}.

In order to prove stability of the optimal controls when they are not explicitly involved in the objective functional, besides (SSOC) an additional structural hypothesis is usually assumed. This situation was studied in \cite{Qui-Wachsmuth2018}, where the authors proved Lipschitz stability of the control  with respect to linear perturbations simultaneously appearing in the state equation and the objective functional. The drawback is that the additional hypothesis is satisfied only by bang-bang controls. Here, we obtain analogous estimates changing the mentioned assumption by a weaker one, see \eqref{E5.2}. Though this second assumption \eqref{E5.2} is stronger than \eqref{E3.13}, it can be satisfied by optimal controls independently if they are bang-bang or not. Moreover our assumption \eqref{E5.2} is satisfied if the (SSOC) and the additional hypothesis are assumed.

Finally, under the assumption \eqref{E5.2}, Lipschitz stability is established for the optimal states with respect to simultaneous perturbations in the equations and in the objective functional with respect to the state and the control, and with respect to the Tikhonov regularization parameter. The stability with respect to the Tikhonov regularization has been studied in \cite{CRT2015} and  \cite{PornerWachsmuth2018}. In \cite{CRT2015}, H\"older stability of the states is proved. In \cite{PornerWachsmuth2018}, stability of the control is proved under (SSOC) and the structural assumption. The reader is also referred to \cite{Daniels2018}, \cite{Wachsmuth-Wachsmuth2011a}, \cite{Wachsmuth-Wachsmuth2011b} for the case of linear partial differential equations.

In this paper, besides providing some new sufficient conditions for Lipschitz stability for the optimal control and associated states, we deal with a semilinear elliptic state equation that is neither monotone nor coercive. Though some crucial results for this state equation are taken from \cite{CMR2020}, some estimates have been proved that are not available in the literature.

The plan of this paper is as follows. In Section \ref{S2}, we analyze the state equation. First, we establish some properties of the linear differential operator of the state equation, and the full semilinear equation is analyzed in the second part of the section. The control problem \Pb is studied in Section \ref{S3}. We prove that our assumption \eqref{E3.13} is a sufficient condition for strong local optimality. Section \ref{S4} is dedicated to the proof of Lipschitz stability of the optimal states. In Section \ref{S5} we introduce the stronger condition \eqref{E5.2} replacing \eqref{E3.13} that allows us to establish the Lipschitz stability of the optimal controls. Finally, in Section \ref{S6}, the Tikhonov regularization is considered.

\section{Analysis of the partial differential equation}
\label{S2}

In this section we analyze the equation \eqref{E1.1}. We split the section in two parts. In the first part, we establish the results concerning the linear operator of the elliptic equation. In the second subsection, the nonlinear equation will be studied.

\subsection{Analysis of the linear differential operator}
\label{S2.1}

We define the differential operator $\mathcal{A}:H_0^1(\Omega) \longrightarrow H^{-1}(\Omega)$ by
\[
\mathcal{A}y = -\dive\big(A(x)\nabla y\big) + b(x)\cdot\nabla y.
\]
The following assumptions are supposed to hold throughout the paper. They ensure that the mathematical objects under consideration are well defined.

\begin{assumption}\label{A1}
	The following statements are fulfilled.
	\begin{itemize}
		\item[(i)] The set $\Omega\subset\mathbb R^n$, $n = 2, 3$, is a bounded domain with a Lipschitz boundary $\Gamma$. 
		The mapping $A:\Omega \longrightarrow \mathbb{R}^{n \times n}$  is measurable and bounded in $\Omega$, and there exists $\Lambda_A>0$ such that $\xi^\top A(x) \xi\ge \Lambda_A|\xi|^2$ for a.e. $x \in \Omega$ and all $\xi\in\mathbb R^n$.
		
		\item[(ii)] We assume that $b \in L^p(\Omega;\mathbb{R}^n)$ with $p \ge 3$ if $n = 3$ and $p > 2$ arbitrary if $n = 2$.
	\end{itemize}
\end{assumption}

Under these assumptions it is known that $\mathcal A:H_{0}^1(\Omega)\to H^{-1}(\Omega)$ is an isomorphism despite the fact that the operator is neither coercive nor monotone; see \cite{CMR2020}, \cite[Theorem 8.3]{Gilbarg-Trudinger83}, \cite{Trudinger73}. The following identity is satisfied
\begin{align*}
\langle \mathcal Ay, z\rangle = \int_{\Omega}A\nabla y\cdot\nabla z\dx+\int_{\Omega}b\cdot\nabla y z\dx\quad \forall y, z \in H_0^1(\Omega),
\end{align*}
where $\langle\cdot,\cdot\rangle$ denotes the duality pairing between $H^{-1}(\Omega)$ and $H_0^1(\Omega)$. 

Along this paper we will set
\[
\|y\|_{H_0^1(\Omega)} = \left(\int_\Omega|\nabla y(x)|^2\dx\right)^{\frac{1}{2}}.
\] 
The next lemma states some properties of $\mathcal{A}$ that will be used later.
\begin{lemma}
The following statements are fulfilled:
\begin{itemize}

	\item[(i)] There exists a constant $C_{\Lambda_A,b}$ such that G\r{a}rding's inequality holds
\begin{equation}
\langle \mathcal Ay, y\rangle \ge \frac{\Lambda_A}{4}\|y\|^2_{H_0^1(\Omega)} - C_{\Lambda_A,b}\|y\|^2_{L^2(\Omega)}\quad \forall y \in H_0^1(\Omega).
\label{E2.1}
\end{equation}
	\item[(ii)] Let $a \in L^\infty(\Omega)$  be a nonnegative function and $h \in H^{-1}(\Omega)$. If $y\in H_0^1(\Omega)$ satisfies $\mathcal Ay + ay = h$ and $h$ is a nonnegative linear form, then $y$ is a nonnegative function as well.	
	\item[(iii)] Let $a$ be as above and $h \in L^r(\Omega)$ with $r > \frac{n}{2}$. Then, the solution $y$ of the above equation belongs to $H_0^1(\Omega)\cap C(\bar\Omega)$. Moreover, there exists a constant $C_r$ independent of $a$ and $h$ such that
\begin{equation}
\|y\|_{H_0^1(\Omega)} + \|y\|_{C(\bar\Omega)}\le C_r\|h\|_{L^r(\Omega)}.
\label{E2.2}\
\end{equation}
\end{itemize}
\label{L2.1}
\end{lemma}
\begin{proof}
The proof of \eqref{E2.1} can be found in \cite{CMR2020}; see also \cite[Lemma 8.4]{Gilbarg-Trudinger83}. For the proof of $(ii)$ the reader is referred again to \cite{CMR2020} and \cite[Theorem 8.1]{Gilbarg-Trudinger83}. The $H_0^1(\Omega)\cap C(\bar{\Omega})$ regularity of $y$ for functions $h\in L^r(\Omega)$ is well known; see \cite[Lemma 8.31]{Gilbarg-Trudinger83}. It remains to prove the estimates \eqref{E2.2} for a constant $C_r$ independent of $h$ and $a$. Let us denote by $y_{a,h} \in H_0^1(\Omega) \cap C(\bar\Omega)$ the solution of  $\mathcal A y + a y = h$.  With $y_{0,h}$ we denote the solution corresponding to $a \equiv 0$. Then, the estimate $\|y_{0,h}\|_{C(\bar\Omega)} \le C\|h\|_{L^r(\Omega)}$ is well known for a constant $C$ depending on $r$, but independent of $h$. Let us write $h = h^+ - h^-$. From $(ii)$ we know that $y_{a,h^+} \ge 0$ and $y_{a,h^-} \ge 0$. Now, since $\mathcal A(y_{a,h^+}-y_{0,h^+})+a(y_{a,h^+}-y_{0,h^+})= -ay_{0,h^+}$, again by item $(ii)$, we obtain $0\le y_{a,h^+} \le y_{0,h^+}$ and consequently $\|y_{a,h^+}\|_{C(\bar\Omega)}\le \|y_{0,h^+}\|_{C(\bar\Omega)}$. Analogously, by the same argument   $0\le y_{a,h^-} \le y_{0,h^-}$ and consequently $\|y_{a,h^-}\|_{C(\bar\Omega)}\le \|y_{0,h^-}\|_{C(\bar\Omega)}$. Therefore,
		\begin{align*}
			\|y_{a,h}\|_{C(\bar\Omega)}&\le \|y_{a,h^+}\|_{ C(\bar\Omega)}+\|y_{a,h^-}\|_{ C(\bar\Omega)}\le \|y_{0,h^+}\|_{ C(\bar\Omega)}+\|y_{0,h^-}\|_{ C(\bar\Omega)}\\
			&\le C\Big(\|h^+\|_{L^{r}(\Omega)}+\|h^-\|_{L^{r}(\Omega)}\Big)\le 2C\|h\|_{L^{r}(\Omega)},
		\end{align*}
where $C$ is independent of $a$ and $h$. To prove the corresponding estimate in $H_0^1(\Omega)$ we use G\r{a}rding's inequality \eqref{E2.1} and the above estimate:
\begin{align*}
&\frac{\Lambda_A}{4}\|y_{a,h}\|_{H^1_0(\Omega)}^2 \le \langle\mathcal{A}y_{a,h},y_{a,h}\rangle + C_{\Lambda_A,b}\|y_{a,h}\|^2_{L^2(\Omega)}\\
& \le \langle\mathcal{A}y_{a,h},y_{a,h}\rangle + \int_\Omega ay^2_{a,h}\dx + C_{\Lambda_A,b}\|y_{a,h}\|^2_{L^2(\Omega)}\\
&=\int_\Omega hy_{a,h}\dx + C_{\Lambda_A,b}\|y_{a,h}\|^2_{L^2(\Omega)} \le |\Omega|^{\frac{r-1}{r}}\|h\|_{L^r(\Omega)}\|y_{a,h}\|_{C(\bar\Omega)} +  C_{\Lambda_A,b}|\Omega|\|y_{a,h}\|^2_{C(\bar\Omega)}\\
& \le 2C\Big(|\Omega|^{\frac{r-1}{r}} + 2CC_{\Lambda_A,b}|\Omega|\Big)\|h\|^2_{L^r(\Omega)},
\end{align*}
where $|\Omega|$ denotes the Lebesgue measure of $\Omega$. Since the above constants are independent of $a$ and $h$, the inequality completes the proof of \eqref{E2.2}. 
\end{proof}

Now, we consider the adjoint operator $\mathcal A^*:H_{0}^1(\Omega)\to H^{-1}(\Omega)$ of $\mathcal A$. Since $\mathcal{A}$ is an isomorphism, $\mathcal{A}^*$ is also an isomorphism. It is obvious that $\mathcal A^*\varphi = -\dive\big(A^\top\nabla\varphi\big)-\dive\big(\varphi b\big)$. The operator $\mathcal{A}^*$ satisfies the same properties established in Lemma \ref{L2.1}. Indeed, the G\r{a}rding's inequality is a consequence of \eqref{E2.1} and the identity $\langle\mathcal{A}^*\varphi,\varphi\rangle = \langle\mathcal{A}\varphi,\varphi\rangle$. The proof of the estimate \eqref{E2.2} is the same for the operator $\mathcal{A}^*$. We only prove the statement $(ii)$. Let $h\in H^{-1}(\Omega)$ be a nonnegative linear form. This means that $\langle h,y\rangle \ge 0$ for every nonnegative function $y \in H_0^1(\Omega)$. Let $\varphi \in H_0^1(\Omega)$ satisfy $\mathcal A^*\varphi+a\varphi = h$. Now, given a nonnegative function $w \in L^2(\Omega)$ we take $y\in H_0^1(\Omega)$ satisfying $\mathcal{A}y + ay = w$. By Lemma \ref{L2.1}-$(ii)$ we have that  $y \ge 0$. Then, we obtain
\begin{align*}
\int_{\Omega}w\varphi\, dx=\langle \mathcal Ay+ ay, \varphi\rangle = \langle \mathcal A^*\varphi+ a\varphi,y \rangle = \langle h,y\rangle \ge 0.
\end{align*}
Since $w$ is an arbitrary nonnegative function of  $L^2(\Omega)$, this inequality yields  $\varphi \ge 0$.

We finish this subsection by proving an $L^s(\Omega)$ estimate.

\begin{lemma}
Assume that $s \in [1,\frac{n}{n - 2})$, $s'$ is its conjugate, and let $a \in L^\infty(\Omega)$ be a nonnegative function. Then, there exists a constant $C_{s'}$ independent of $a$ such that
\begin{equation}
\left\{\begin{array}{l}\displaystyle\|y_h\|_{L^s(\Omega)} \le C_{s'}\|h\|_{L^1(\Omega)},\\\displaystyle\|\varphi_h\|_{L^s(\Omega)} \le C_{s'}\|h\|_{L^1(\Omega)},\end{array}\right.\quad \forall h \in H^{-1}(\Omega) \cap L^1(\Omega),
\label{E2.3}
\end{equation}
where $y_h$ and $\varphi_h$ satisfy the equations $\mathcal Ay_h+ ay_h = h$ and $\mathcal A^*\varphi_h+ a\varphi_h = h$, respectively, and $C_{s'}$ is given by \eqref{E2.2} with $r = s'$.
\label{L2.2}
\end{lemma}
\begin{proof} 
We prove the estimate \eqref{E2.3} for $\varphi_h$ and $n=3$, the proof being identical for $y_h$ and analogous for $n=2$ with minor modifications. First we observe that $H_0^1(\Omega) \subset L^{6}(\Omega) \subset L^{3}(\Omega)$, hence $\varphi_h \in L^s(\Omega)$. As a consequence we obtain that $|\varphi_h|^{s - 1}\text{sign}(\varphi_h) \in L^{s'}(\Omega)$. Moreover, $s < 3$ implies that $s' > \frac{3}{2}$. According to Lemma \ref{L2.1}-$(iii)$, the solution of $\mathcal{A}y + ay = |\varphi_h|^{s - 1}\text{sign}(\varphi_h)$ belongs to $H_0^1(\Omega) \cap C(\bar\Omega)$ and satisfies $\|y\|_{C(\bar\Omega)} \le C_{s'}\||\varphi_h|^{s - 1}\text{sign}(\varphi_h)\|_{L^{s'}(\Omega)} = C_{s'}\|\varphi_h\|^{s - 1}_{L^s(\Omega)}$, where $C_{s'}$ is independent of $a$ and $h$. Using these facts we infer
\begin{align*}
&\|\varphi_h\|^s_{L^s(\Omega)} = \int_{\Omega}|\varphi_h|^s\dx=\langle \mathcal Ay+ay,\varphi_h\rangle = \langle \mathcal A^*\varphi_h+a\varphi_h, y\rangle\\
& = \int_{\Omega}hy\dx \le \|h\|_{L^1(\Omega)}\|y\|_{C(\bar\Omega)} \le C_{s'}\|h\|_{L^1(\Omega)}\|\varphi_h\|^{s - 1}_{L^s(\Omega)}.
\end{align*}
This proves \eqref{E2.3} for $\varphi_h$. 
\end{proof}

\subsection{Analysis of the semilinear equation}

In this subsection, we formulate some results concerning the semilinear equation \eqref{E1.1}. For this purpose we make the following assumptions on the nonlinear term of the equation.

\begin{assumption}\label{A2}
We assume that $f:\Omega \times \mathbb{R}\longrightarrow \mathbb{R}$ is a Carath\'eodory function of class $C^2$ with respect to the second variable satisfying:
\begin{align}
&f(\cdot,0)\in L^r(\Omega)  \text{ with } r> \frac{n}{2} \text{ and } \frac{\partial f}{\partial y}(x,y) \geq 0 \  \forall y \in \mathbb{R},
\label{E2.4}\\
&\forall M>0\ \exists C_{f,M}>0 \text{ such that }\left|\frac{\partial f}{\partial y}(x,y)\right|+
\left|\frac{\partial^2 f}{\partial y^2}(x,y)\right| \leq C_{f,M} \ \forall |y| \leq M,
\label{E2.5}\\
&\left\{\begin{array}{l}\displaystyle\forall M > 0 \text{ and } \forall \varepsilon > 0\ \exists \delta > 0 \text{ such that}\vspace{1mm}\\\displaystyle
\left|\frac{\partial^2f}{\partial y^2}(x,y_2) - \frac{\partial^2f}{\partial y^2}(x,y_1)\right| < \varepsilon \mbox{ if } |y_1|, |y_2| \leq M \text{ and }  |y_2 - y_1| \le \delta,\end{array}\right.
\label{E2.6}
\end{align}
for almost every $x \in \Omega$.
\end{assumption}

\begin{theorem}
	Let Assumptions \ref{A1} and \ref{A2}  hold. If $u$ belongs to $L^r(\Omega)$ for some $r>n/2$, then there exists a unique solution $y_u\in H_0^1(\Omega)\cap C(\bar\Omega)$ of \eqref{E1.1}. Moreover, there exists a constant $K_{f,r}$ independent of $u$ such that
\begin{equation}
\|y_u\|_{H_0^1(\Omega)} + \|y_u\|_{C(\bar\Omega)} \le K_{f,r}\big(\|u\|_{L^r(\Omega)} + \|f(\cdot,0)\|_{L^r(\Omega)} + 1\big).
\label{E2.7}
\end{equation}
Further, if $\{u_k\}_{k = 1}^\infty$ is a sequence converging weakly to $u$ in $L^r(\Omega)$, then $y_{u_k} \to y_u$ strongly in $H_0^1(\Omega) \cap C(\bar\Omega)$.
\label{T2.1}
\end{theorem}

The reader is referred to \cite{CMR2020} for the proof of this result. As a consequence of \eqref{E2.7} we get
\begin{equation}
\exists K_U > 0 \text{ such that } \|y_u\|_{H_0^1(\Omega)} + \|y_u\|_{C(\bar\Omega)} \le K_U\quad \forall u \in \Uad.
\label{E2.8}
\end{equation}

For each $r>n/2$, we define the map $G_r:L^r(\Omega)\to H_0^1(\Omega)\cap C(\bar\Omega)$ by $G_r(u)=y_u$.

\begin{theorem}
		Let Assumptions \ref{A1} and \ref{A2} hold. For every $r > \frac{n}{2}$ the map $G_r$ is of class $C^2$, and the first and second derivatives at $u\in L^r(\Omega)$ in the directions $v, v_1, v_2\in L^r(\Omega)$, denoted by $z_{u,v} = G'_r(u)v$ and $z_{u,v_1,v_2} = G''_r(u)(v_1,v_2)$, are the solutions of the equations 
\begin{align}
&\mathcal Az+\frac{\partial f}{\partial y}(x,y_u)z = v,\label{E2.9}\\
&\mathcal Az+\frac{\partial f}{\partial y}(x,y_u)z = -\frac{\partial^2f}{\partial y^2}(x,y_u)z_{u,v_1}z_{u,v_2},\label{E2.10}
\end{align}
respectively.
\label{T2.2}
\end{theorem}

The proof of this theorem is an easy application of the implicit function theorem; see \cite{CMR2020}.

\begin{lemma}
The following statements are fulfilled.
\begin{itemize}
\item[(i)] Suppose that $r > \frac{n}{2}$ and $s \in [1,\frac{n}{n - 2})$. Then, there exist constants $K_r$ depending on $r$ and $M_s$ depending on $s$ such that for every $u, \bar u \in \Uad$
\begin{align}
&\|y_u - y_{\bar u} - z_{\bar u,u - \bar u}\|_{ C(\bar\Omega)}\le K_r\|y_u - y_{\bar u}\|^2_{L^{2r}(\Omega)},
\label{E2.11}\\
&\|y_u - y_{\bar u} - z_{\bar u,u - \bar u}\|_{L^{s}(\Omega)}\le M_s\|y_u-y_{\bar u}\|^2_{L^{2}(\Omega)}.
\label{E2.12}
\end{align}
\item[(ii)] Taking $C_X = K_2\sqrt{|\Omega|}$ if $X = C(\bar\Omega)$ and $C_X = M_2$ if $X = L^2(\Omega)$, the following inequality holds
\begin{equation}
\|z_{u,v} - z_{\bar u,v}\|_X \le C_X\|y_u - y_{\bar u}\|_X\|z_{\bar u,v}\|_X\ \forall u,\bar u \in \Uad \text{ and } \forall v \in L^2(\Omega).
\label{E2.13}
\end{equation}
\item[(iii)] Let $X$ be as in $(ii)$. There exists $\varepsilon > 0$ such that for all $\bar u, u \in \Uad$ with $\|y_{u} - y_{\bar u}\|_{C(\bar\Omega)} \le \varepsilon$ the following inequalities are satisfied
\begin{align}
&\frac{1}{2}\|y_{u} - y_{\bar u}\|_X \le \|z_{\bar u,u - \bar u}\|_X \le \frac{3}{2}\|y_{u} - y_{\bar u}\|_X,
\label{E2.14}\\
&\frac{1}{2}\|z_{\bar u,v}\|_X \le \|z_{u,v}\|_X \le \frac{3}{2}\|z_{\bar u,v}\|_X \quad \forall v \in L^2(\Omega).\label{E2.15}
\end{align}
\end{itemize}
\label{L2.3}
\end{lemma}	
\begin{proof}
Let us set $\phi = y_u - y_{\bar u} - z_{\bar u,u - \bar u} \in H_0^1(\Omega) \cap C(\bar\Omega)$. From the equations satisfied by the three functions and using the mean value theorem we get 
\[
\mathcal{A}\phi + \frac{\partial f}{\partial y}(x,y_{\bar u})\phi =  \Big[\frac{\partial f}{\partial y}(x,y_{\bar u}) - \frac{\partial f}{\partial y}(x,y_\theta)\Big](y_u - y_{\bar u}),
\]
where $y_\theta(x) = y_{\bar u}(x) + \theta(x)(y_u(x) - y_{\bar u}(x))$ with $\theta:\Omega \longrightarrow [0,1]$ measurable. Using again the mean value theorem we deduce
\[
\mathcal{A}\phi + \frac{\partial f}{\partial y}(x,y_{\bar u})\phi =  -\theta\frac{\partial^2f}{\partial y^2}(x,y_\vartheta)(y_u - y_{\bar u})^2
\]
with $y_\vartheta(x) = y_{\bar u}(x) + \vartheta(x)(y_\theta(x) - y_{\bar u}(x))$ and $\vartheta:\Omega \longrightarrow [0,1]$ measurable. By Lemma \ref{L2.1}-$(iii)$ and taking into account \eqref{E2.5} and \eqref{E2.8} we infer the existence of $C_r$ independent of $u, \bar u \in \Uad$ such that
\[
\|\phi\|_{C(\bar\Omega)} \le C_rC_{f,K_U}\|(y_{u}-y_{\bar u})^2\|_{L^r(\Omega)} = C_rC_{f,K_U}\|y_{u}-y_{\bar u}\|^2_{L^{2r}(\Omega)},
\]
which proves \eqref{E2.11} with $K_r = C_rC_{f,K_U}$. To prove \eqref{E2.12} we use Lemma \ref{L2.2} to obtain
\[
\|\phi\|_{L^{s}(\Omega)}\le C_{s'}C_{f,K_U}\|(y_u-y_{\bar u})^2\|_{L^1(\Omega)} =  C_{s'}C_{f,K_U}\|y_u-y_{\bar u}\|^2_{L^2(\Omega)}.
\]
Taking $M_s =  C_{s'}C_{f,K_U}$, \eqref{E2.12} follows.

Now we prove \eqref{E2.13} for $X = C(\bar\Omega)$. Setting $\psi = z_{u,v} - z_{\bar u,v}$ and subtracting the corresponding equations we infer with the mean value theorem
\[
\mathcal{A}\psi + \frac{\partial f}{\partial y}(x,y_{u})\psi = \Big[\frac{\partial f}{\partial y}(x,y_{\bar u}) - \frac{\partial f}{\partial y}(x,y_{u})\Big]z_{\bar u,v} = \frac{\partial^2f}{\partial y^2}(x,y_\theta)(y_{\bar u} - y_u)z_{\bar u,v}.
\]
Taking $r = 2$ in \eqref{E2.2} and using \eqref{E2.5} and \eqref{E2.8} it follows from the above equation
\[
\|\psi\|_{C(\bar\Omega)} \le C_2C_{f,K_U}\|(y_{\bar u} - y_u)z_{\bar u,v}\|_{L^2(\Omega)} \le K_2\sqrt{|\Omega|}\|y_u - y_{\bar u}\|_{C(\bar\Omega)}\|z_{\bar u,v}\|_{C(\bar\Omega)},
\]
which proves \eqref{E2.13} for $X = C(\bar\Omega)$. The proof for $X = L^2(\Omega)$ is analogous, we use the estimate \eqref{E2.3} for $s = 2$ instead of \eqref{E2.2}.

To prove \eqref{E2.14} for $X=C(\bar\Omega)$ we use \eqref{E2.11} with $r = 2$ to get
\begin{align*}
&\|y_{u} - y_{\bar u}\|_{C(\bar\Omega)} \le \|\phi\|_{C(\bar\Omega)} + \|z_{\bar u,u - \bar u}\|_{C(\bar\Omega)} \le K_2\|y_{u} - y_{\bar u}\|^2_{L^4(\Omega)} + \|z_{\bar u,u - \bar u}\|_{C(\bar\Omega)}\\
&\le K_2\sqrt{|\Omega|}\|y_{u} - y_{\bar u}\|^2_{C(\bar\Omega)} + \|z_{\bar u,u - \bar u}\|_{C(\bar\Omega)}.
\end{align*}
Choosing $\varepsilon_1 = [2K_2\sqrt{|\Omega|}]^{-1}$ the first inequality of \eqref{E2.14} follows if $\|y_{u} - y_{\bar u}\|_{C(\bar\Omega)} < \varepsilon_1$. To deal with the case $X = L^2(\Omega)$ we use \eqref{E2.12} with $s = 2$ and obtain
\begin{align*}
&\|y_{u} - y_{\bar u}\|_{L^2(\Omega)} \le \|\phi\|_{L^2(\Omega)} +  \|z_{\bar u,u - \bar u}\|_{L^2(\Omega)} \le M_2\|y_{u} - y_{\bar u}\|^2_{L^2(\Omega)} +  \|z_{\bar u,u - \bar u}\|_{L^2(\Omega)}\\
&\le M_2\sqrt{|\Omega|}\|y_{u} - y_{\bar u}\|_{C(\bar\Omega)}\|y_{u} - y_{\bar u}\|_{L^2(\Omega)} +  \|z_{\bar u,u - \bar u}\|_{L^2(\Omega)}.
\end{align*}
Hence, taking $\varepsilon_2 =  [2M_2\sqrt{|\Omega|}]^{-1}$ we obtain the first inequality of \eqref{E2.14} with $X = L^2(\Omega)$ if $\|y_{u} - y_{\bar u}\|_{C(\bar\Omega)} < \varepsilon_2$.

To prove the second inequality of \eqref{E2.14} for $X = C(\bar\Omega)$, we proceed as follows
\begin{align*}
&\|z_{\bar u,u - \bar u}\|_{C(\bar\Omega)} \le \|\phi\|_{C(\bar\Omega)} + \|y_{u} - y_{\bar u}\|_{C(\bar\Omega)}\le  K_2\sqrt{|\Omega|}\|y_{u} - y_{\bar u}\|^2_{C(\bar\Omega)} +  \|y_{u} - y_{\bar u}\|_{C(\bar\Omega)}\\
&\le \frac{3}{2} \|y_{u} - y_{\bar u}\|_{C(\bar\Omega)}\quad\text{ if } \|y_{u} - y_{\bar u}\|_{C(\bar\Omega)} < \varepsilon_1.
\end{align*}
Similarly the second inequality of \eqref{E2.14} follows if $X = L^2(\Omega)$ with $\varepsilon_2$ replacing $\varepsilon_1$.

Finally, we prove \eqref{E2.15}. Using \eqref{E2.13} we obtain
\begin{align*}
&\|z_{u,v}\|_X \le \|z_{u,v} - z_{\bar u,v}\|_X + \|z_{\bar u,v}\|_X \le C_X\|y_{u} - y_{\bar u}\|_X\|z_{\bar u,v}\|_X + \|z_{\bar u,v}\|_X,\\
&\|z_{\bar u,v}\|_X \le \|z_{u,v} - z_{\bar u,v}\|_X + \|z_{u,v}\|_X \le C_X\|y_{u} - y_{\bar u}\|_X\|z_{\bar u,v}\|_X + \|z_{u,v}\|_X.
\end{align*}
Therefore, selecting $\varepsilon = [2C_2]^{-1}$ for $X=C(\bar\Omega)$ and $\varepsilon = [2C_2\sqrt{|\Omega|}]^{-1}$ for $X=L^2(\Omega)$, \eqref{E2.15} follows if $\|y_{u} - y_{\bar u}\|_{C(\bar\Omega)} \le \varepsilon$.
\hfill\end{proof}

\section{The Control Problem}
\label{S3}

In this section, we make assumptions on the objective functional $J$ so that \Pb has at least one solution and the first and second order conditions for local optimality can be established. Since the problem is not convex, we will consider not only global minimizers, but also local minimizers. Throughout this paper, we will say that $\bar u$ is a local minimizer of \Pb if $\bar u \in \Uad$ and there exists a ball $B_\rho(\bar u) \subset L^2(\Omega)$ such that $J(\bar u) \le J(u)$ for every $u \in \Uad \cap B_\rho(\bar u)$. We will also say that $\bar u$ is a strong local minimizer of \Pb if $\bar u \in \Uad$ and there exists $\varepsilon > 0$ such that $J(\bar u) \le J(u)$ for every $u \in \Uad$ with $\|y_u - y_{\bar u}\|_{C(\bar\Omega)} < \varepsilon$. If the previous inequalities are strict whenever $u \neq \bar u$, then we say that $\bar u$ is a strict (strong) local minimizer. As far as we know, the notion of strong local minimizers in the framework of control of partial differential equations was introduced for the first time in \cite{BBS2014}; see also \cite{BS2016}.

We make the following assumptions on $L$.

\begin{assumption}
The function $L:\Omega \times \mathbb{R}^2 \longrightarrow \mathbb{R}$ is Carath\'eodory and of class $C^2$ with respect to the second variable. In addition, we assume that
\begin{align}
&L(x,y,u) = L_0(x,y) + g(x)u \ \text{ with } \ L_0(\cdot,0) \in L^1(\Omega)\ \text{ and }\ g \in L^\infty(\Omega),\label{E3.1}\\
&\left\{\begin{array}{l}\forall M > 0 \ \exists \psi_M \in L^2(\Omega) \text{ and } C_{L,M} > 0 \text{ such that }\vspace{1mm}\\
\displaystyle \Big|\frac{\partial L}{\partial y}(x,y,u)\Big| \le \psi_M(x)\ \text{ and } \ \ \Big|\frac{\partial^2L}{\partial y^2}(x,y,u)\Big| \le C_{L,M}\ \forall |y| \le M,\end{array}\right.\label{E3.2}\\
&\left\{\begin{array}{l}\forall M > 0 \text{ and } \forall \varepsilon > 0\ \exists \delta > 0 \text{ such that }\vspace{1mm}\\
\displaystyle\left|\frac{\partial^2L}{\partial y^2}(x,y_2,u) - \frac{\partial^2L}{\partial y^2}(x,y_1,u)\right| < \varepsilon \mbox{ if } |y_1|, |y_2| \leq M,\ |y_2 - y_1| \le \delta,\end{array}\right.
\label{E3.3}
\end{align}
for almost every $x \in \Omega$.
\label{A3}
\end{assumption}

Using Theorem \ref{T2.1}, the assumptions on $L$, and the boundedness of $\Uad$ in $L^\infty(\Omega)$, the existence of at least one solution of \Pb follows. Indeed, if we take a minimizing sequence $\{u_k\}_{k = 1}^\infty$, we can assume that $u_k \stackrel{*}{\rightharpoonup} \bar u$ in $L^\infty(\Omega)$. Then Theorem \ref{T2.1} implies that $y_{u_k} \to y_{\bar u}$ strongly in $H_0^1(\Omega) \cap C(\bar\Omega)$. Further, using \eqref{E2.8} and \eqref{E3.2} with $M = K_U$ we infer with the mean value theorem
\[
|L_0(x,y_{u_k}(x))| \le |L_0(x,0)| + \psi_{K_U}(x)K_U.
\]
Then we can apply Lebesgue's dominated convergence theorem to pass to the limit in the objective functional and to obtain $J(u_k) \to J(\bar u)$.

 In order to derive the first order optimality conditions satisfied by a local minimizer we address the issue of the differentiability of the objective functional $J$.

\begin{theorem}
Suppose that $r > \frac{n}{2}$. Then, the functional $J:L^r(\Omega) \longrightarrow \mathbb{R}$ is of class $C^2$. Moreover, given $u, v, v_1, v_2 \in L^r(\Omega)$ we have
\begin{align}
&J'(u)v = \int_\Omega(\varphi_u + g)v\dx,\label{E3.4}\\
&J''(u)(v_1,v_2) = \int_\Omega\Big[\frac{\partial^2L}{\partial y^2}(x,y_u,u) - \varphi_u\frac{\partial^2f}{\partial y^2}(x,y_u)\Big]z_{u,v_1}z_{u,v_2}\dx,\label{E3.5}
\end{align}
where $\varphi_u \in H_0^1(\Omega) \cap C(\bar\Omega)$ is the unique solution of the adjoint equation
\begin{equation}
\left\{\begin{array}{l}\displaystyle \mathcal{A}^*\varphi + \frac{\partial f}{\partial y}(x,y_u)\varphi = \frac{\partial L}{\partial y}(x,y_u,u) \text{ in } \Omega,\\ \varphi = 0\text{ on } \Gamma.\end{array}\right.
\label{E3.6}
\end{equation}
\label{T3.1}
\end{theorem}
This is a straightforward consequence of Theorem \ref{T2.2}, Assumption \ref{A3}, and the chain rule. The only critical issue is the existence, uniqueness, and regularity of $\varphi_u$. But this is an immediate consequence of Lemma \ref{L2.1}-(iii) that, as already mentioned, applies to the operator $\mathcal{A}^*$ as well. From this theorem, the optimality conditions follow in the classical way.

\begin{theorem}
Let $\bar u$ be a (strong or not) local minimizer of \Pb, then there exist two unique elements $\bar y, \bar\varphi \in H_0^1(\Omega) \cap C(\bar\Omega)$ such that
\begin{align}
& \left\{\begin{array}{l} \mathcal{A}\bar y + f(x,\bar y) = \bar u \text{ in } \Omega,\\ \bar y = 0\text{ on } \Gamma,\end{array}\right.
\label{E3.7}\\
&\left\{\begin{array}{l}\displaystyle \mathcal{A}^*\bar\varphi + \frac{\partial f}{\partial y}(x,\bar y)\bar\varphi = \frac{\partial L}{\partial y}(x,\bar y,\bar u) \text{ in } \Omega,\\ \bar\varphi = 0\text{ on } \Gamma,\end{array}\right.
\label{E3.8}\\
&\int_\Omega(\bar\varphi + g)(u - \bar u)\, dx \ge 0 \quad \forall u \in \Uad. \label{E3.9}
\end{align}
\label{T3.2}
\end{theorem}

The derivation of sufficient second order conditions for local optimality is more delicate. First we introduce the cone of critical directions on which we formulate the necessary second order conditions for optimality: if $\bar u \in \Uad$ is a local minimizer of \Pb we define
\[
 C_{\bar u} = \{v \in L^2(\Omega) : J'(\bar u)v = 0\ \text{ and $v$ satisfies the sign conditions \eqref{E3.10}}\},
\]
\begin{equation}
v(x)\left\{\begin{array}{cl}\ge 0&\text{if } \bar u(x) = u_a,\\\le 0&\text{if } \bar u(x) = u_b.\end{array}\right.
\label{E3.10}
\end{equation}
As usual, from \eqref{E3.9} we deduce that $(\bar\varphi+g)(x)v(x) \ge 0$ for almost all $x \in \Omega$ if $v \in L^2(\Omega)$ satisfies \eqref{E3.10}. Therefore, the condition $J'(\bar u)v = 0$ for $v$ satisfying \eqref{E3.10} is only possible if $v(x) = 0$ for almost every $x \in \Omega$ such that $(\bar\varphi+g)(x) \neq 0$. Therefore, $C_{\bar u}$ can be written
\[
 C_{\bar u} = \{v \in L^2(\Omega) : \text{satisfying \eqref{E3.10} and } v(x) = 0 \text{ if } |(\bar\varphi+g)(x)| > 0\}.
\]

It is well known that every local minimizer $\bar u$ satisfies the second order necessary optimality condition: $J''(\bar u)v^2 \ge 0$ for all $v \in C_{\bar u}$; see, for instance, \cite{Casas-Troltzsch2012}. However, based on $C_{\bar u}$ it is not possible to get sufficient second order conditions for local optimality. The reader is referred to \cite{Dunn98} for a counterexample. A procedure suggested by several authors consists in extending the cone of critical directions $C_{\bar u}$; see \cite{DHPY95,Dunn95,Maurer81,Maurer-Zowe79}. Two possible extensions of $C_{\bar u}$ seem natural after the above comments: for $\tau > 0$ we define the extended cones
\begin{align*}
&D^\tau_{\bar u} = \{v \in L^2(\Omega) : \text{satisfying \eqref{E3.10} and } v(x) = 0 \text{ if } |(\bar\varphi + g)(x)| > \tau\},\\
&G^\tau_{\bar u} = \{v \in L^2(\Omega) : \text{satisfying \eqref{E3.10} and }  J'(\bar u)v \le \tau\|z_v\|_{L^1(\Omega)}\}.
\end{align*}
On any of these cones we can formulate sufficient second order conditions for local optimality. Obviously, both are extensions of $C_{\bar u}$. In \cite{Casas-Mateos2020}, the authors introduced the cone $C_{\bar u}^\tau = D^\tau_{\bar u} \cap G^\tau_{\bar u}$, which is also an extension of $C_{\bar u}$. They proved that the first order optimality conditions \eqref{E3.7}--\eqref{E3.9} along with the condition
\begin{equation}
\exists \delta > 0 \text{ such that } J''(\bar u)v^2 \ge \delta\|z_v\|^2_{L^2(\Omega)}\quad \forall v \in C_{\bar u}^\tau
\label{E3.11}
\end{equation}
imply the existence of $\kappa > 0$ and $\varepsilon > 0$ such that
\begin{equation}
J(\bar u) + \frac{\kappa}{2}\|y_u - \bar y\|^2_{L^2(\Omega)} \le J(u) \ \ \forall u \in \Uad \text{ such that } \|y_u - \bar y\|_{C(\bar\Omega)} < \varepsilon.
\label{E3.12}
\end{equation}
Actually, the proof of \cite{Casas-Mateos2020} was carried out for a parabolic control problem with $g = 0$. However, the same proof works for the elliptic case and $g \neq 0$. Here, we formulate a new assumption leading to the same result \eqref{E3.12} as \eqref{E3.11} does.

\begin{assumption}
There exist numbers $\alpha > 0$ and $\gamma > 0$ such that
\begin{equation}
J'(\bar u)(u - \bar u) + J''(\bar u)(u - \bar u)^2 \ge \gamma\|z_{\bar u,u - \bar u}\|^2_{L^2(\Omega)}\ \forall u \in \Uad \text{ with } \|y_u - \bar y\|_{C(\bar\Omega)} < \alpha.
\label{E3.13}
\end{equation}
\label{A4}
\end{assumption}
It was proved in \cite{Casas-Mateos2021} that \eqref{E3.11} implies \eqref{E3.13}. Therefore, \eqref{E3.13} appears as a weaker assumption. However, the next theorem proves that it is sufficient to imply \eqref{E3.12}.

\begin{theorem}
Let $\bar u \in \Uad$ satisfy the optimality conditions \eqref{E3.7}--\eqref{E3.9} and Assumption \ref{A4}. Then,  there exist $\varepsilon > 0$ and $\kappa > 0$ such that \eqref{E3.12} holds.
\label{T3.3}
\end{theorem}

Before proving this theorem we establish some lemmas.

\begin{lemma}
Let $\bar u \in \Uad$ be fixed with associated state $\bar y$. Then, the following inequality holds for all $\theta \in [0,1]$ and $u \in \Uad$
\begin{align}
\|y_{\bar u + \theta(u - \bar u)} - \bar y\|_{C(\bar\Omega)} &\le  (C_2C_{f,K_U}\sqrt{|\Omega|}\|y_u - \bar y\|_{C(\bar\Omega)} + 1)\|y_u - \bar y\|_{C(\bar\Omega)},
\label{E3.14}
\end{align}
where $C_2$ is the constant of \eqref{E2.2} with $r = 2$ and $C_{f,K_U}$ is the one deduced from \eqref{E2.5} and \eqref{E2.8}.
\label{L3.1}
\end{lemma}

\begin{proof}
The proof  of this lemma is based on the analogous result for parabolic control problems established in \cite{CMR2019}. We take $\theta \in [0,1]$ and $u \in \Uad$. We set $\phi = y_{\bar u + \theta(u - \bar u)} - [\bar y + \theta(y_u - \bar y)]$. Then, we have
\[
\mathcal{A}\phi + f(x,y_{\bar u + \theta(u - \bar u)}) -  [f(x,\bar y) + \theta(f(x,y_u) - f(x,\bar y))] = 0.
\]
Applying the mean value theorem, we obtain measurable functions $\theta_i:\Omega \longrightarrow [0,1]$, $i = 1, 2$, such that
\begin{align*}
&f(x,y_{\bar u + \theta(u - \bar u)}) -  f(x,\bar y) = \frac{\partial f}{\partial y}(x,y_1)(y_{\bar u + \theta(u - \bar u)} - \bar y)\,\text{and}\, y_1 = \bar y + \theta_1(y_{\bar u + \theta(u - \bar u)} - \bar y),\\
&f(x,y_u) - f(x,\bar y) = \frac{\partial f}{\partial y}(x,y_2)(y_u - \bar y) \text{ with } y_2 = \bar y + \theta_2(y_u - \bar y).
\end{align*}
Inserting these identities in the above partial differential equation we infer
\[
\mathcal{A}\phi + \frac{\partial f}{\partial y}(x,y_1)(y_{\bar u + \theta(u - \bar u)} - \bar y) - \theta\frac{\partial f}{\partial y}(x,y_2)(y_u - \bar y) = 0.
\]
Noting that $y_{\bar u + \theta(u - \bar u)} - \bar y = \phi + \theta(y_u - \bar y)$, the above equality and a new application of the mean value theorem lead to
\[
\mathcal{A}\phi + \frac{\partial f}{\partial y}(x,y_1)\phi = \theta\Big[\frac{\partial f}{\partial y}(x,y_2) - \frac{\partial f}{\partial y}(x,y_1)\Big](y_u - \bar y) = \theta \frac{\partial^2f}{\partial y^2}(x,y_3)(y_u - \bar y)^2,
\]
where $y_3 = y_1 + \theta_3(y_2 - y_1)$. Using \eqref{E2.2} with $r = 2$, \eqref{E2.5}, and \eqref{E2.8} we infer
\[
\|\phi\|_{C(\bar\Omega)} \le C_2C_{f,K_U}\|(y_u - \bar y)^2\|_{L^2(\Omega)} \le C_2C_{f,K_U}\sqrt{|\Omega|}\|y_u - \bar y\|^2_{C(\bar\Omega)}.
\]
This implies
\begin{align*}
&\|y_{\bar u + \theta(u - \bar u)} - \bar y\|_{C(\bar\Omega)}= \|\phi + \theta(y_u - \bar y)\|_{C(\bar\Omega)}\\
&\le (C_2C_{f,K_U}\sqrt{|\Omega|}\|y_u - \bar y\|_{C(\bar\Omega)} + 1)\|y_u - \bar y\|_{C(\bar\Omega)}.
\end{align*}
\end{proof}

\begin{lemma}
There exists a constant $M_U > 0$ such that
\begin{equation}
\|\varphi_u\|_{C(\bar\Omega)} \le M_U\quad \forall u \in \Uad.
\label{E3.15}
\end{equation}
Moreover, given $\bar u \in \Uad$ with associated state $\bar y$ and adjoint state $\bar\varphi$, we have
\begin{equation}
\|\varphi_{\bar u + \theta(u - \bar u)} - \bar\varphi\|_{C(\bar\Omega)} \le C\|y_u - \bar y\|_{C(\bar\Omega)} \quad \forall \theta \in [0,1] \text{ and } \ \forall u \in \Uad,
\label{E3.16}
\end{equation}
where $C$ depends only on $f$, $L$, $\Uad$, and $\Omega$.
\label{L3.2}
\end{lemma}
\begin{proof}
For the proof of \eqref{E3.15} we use \eqref{E2.2} with $r = 2$, \eqref{E2.8}, and \eqref{E3.2} as follows
\[
\|\varphi_u\|_{C(\bar\Omega)} \le C_2\Big\|\frac{\partial L}{\partial y}(x,y_u,u)\Big\|_{L^2(\Omega)} \le M_U = C_2\|\psi_{K_U}\|_{L^2(\Omega)}.
\]

Let us prove \eqref{E3.16}. Given $u \in \Uad$ and $\theta \in [0,1]$ let us denote $u_\theta = \bar u + \theta(u - \bar u)$, $y_\theta = y_{u_\theta}$, and $\varphi_\theta = \varphi_{u_\theta}$. Subtracting the equations satisfied by $\varphi_\theta$ and $\bar\varphi$ we get with the mean value theorem
\begin{align*}
&\mathcal A^ *(\varphi_\theta - \bar\varphi) + \frac{\partial f}{\partial y}(x,\bar y)(\varphi_\theta - \bar\varphi) = \frac{\partial L}{\partial y}(x,y_\theta,u_\theta) - \frac{\partial L}{\partial y}(x,\bar y,\bar u)\\
&+ \Big[\frac{\partial f}{\partial y}(x,\bar y) - \frac{\partial f}{\partial y}(x,y_\theta)\Big]\varphi_\theta = \Big[\frac{\partial^2L}{\partial y^2}(x,y_\vartheta,u_\vartheta) - \varphi_\theta\frac{\partial^2f}{\partial y^2}(x,y_\vartheta)\Big](y_\theta - \bar y),
\end{align*}
where $y_\vartheta = \bar y + \vartheta(y_\theta - \bar y)$ for some measurable function $\vartheta:\Omega \longrightarrow [0,1]$. Now, we apply \eqref{E2.2} with $r = 2$, \eqref{E2.8}, \eqref{E3.15}, \eqref{E2.5}, and \eqref{E3.2} to get from the above equation
\[
\|\varphi_\theta - \bar\varphi\|_{C(\bar\Omega)} \le C_2(C_{L,K_U} + M_UC_{f,K_U})\sqrt{|\Omega|}\|y_\theta - \bar y\|_{C(\bar\Omega)}.
\]
Then, \eqref{E3.16} follows from Lemma \ref{L3.1}.
\end{proof}
\begin{lemma}
For every $\rho > 0$ there exists $\varepsilon > 0$ such that if $u \in \Uad$ and $\|y_u - \bar y\|_{C(\bar\Omega)} < \varepsilon$ then
\begin{equation}
|[J''(\bar u + \theta(u - \bar u)) - J''(\bar u)]v^2| < \rho\|z_{\bar u,v}\|^2_{L^2(\Omega)}\quad \forall v \in L^2(\Omega) \text{ and } \forall \theta \in [0,1].
\label{E3.17}
\end{equation}
\label{L3.3}
\end{lemma}

\begin{proof}
First, let us denote $u_\theta$, $y_\theta$, and $\varphi_\theta$ as in the proof of Lemma \ref{L3.2}. From \eqref{E3.5} we get
\begin{align*}
&|[J''(\bar u + \theta(u - \bar u)) - J''(\bar u)]v^2| \le \int_\Omega\Big|\Big[\frac{\partial^2L}{\partial y^2}(x,y_\theta,u_\theta) - \frac{\partial^2L}{\partial y^2}(x,\bar y,\bar u)\Big]z_{u_\theta,v}^2\Big|\dx\\
&+ \int_\Omega\Big|(\varphi_\theta - \bar\varphi)\frac{\partial^2f}{\partial y^2}(x,y_\theta)z_{u_\theta,v}^2\Big|\dx+ \int_\Omega\Big|\bar\varphi\Big[\frac{\partial^2f}{\partial y^2}(x,y_\theta) - \frac{\partial^2f}{\partial y^2}(x,\bar y)\Big]z_{u_\theta,v}^2\Big|\dx\\
&+ \int_\Omega\Big|\Big[\frac{\partial^2L}{\partial y^2}(x,\bar y,\bar u) - \bar\varphi\frac{\partial^2f}{\partial y^2}(x,\bar y)\Big](z^2_{u_\theta,v} - z^2_{\bar u,v})\Big|\dx\\
& = I_1 + I_2 + I_3 + I_4.
\end{align*}
Let us estimate the terms $I_i$. For $I_1$ we deduce from \eqref{E3.3}, \eqref{E2.15}, and \eqref{E3.14} that for every $\rho > 0$ there exists $\varepsilon > 0$ such that $I_1 \le \rho\|z_{\bar u,v}\|^2_{L^2(\Omega)}$ if  $\|y_u - \bar y\|_{C(\bar\Omega)} < \varepsilon$.
The same estimate can be deduced for $I_2$ using \eqref{E2.5}, \eqref{E2.8}, \eqref{E2.15},  and \eqref{E3.16}. The estimate for $I_3$ follows from \eqref{E2.6}, \eqref{E2.8}, \eqref{E2.15}, \eqref{E3.14}, and \eqref{E3.15}. Finally, we estimate $I_4$ by using \eqref{E2.5}, \eqref{E2.8}, \eqref{E2.13}, \eqref{E2.15}, \eqref{E3.2}, \eqref{E3.14}, and \eqref{E3.15} to infer that
\begin{align*}
&I_4 \le (C_{L,K_U} + M_UC_{f,K_U})\|z_{u_\theta,v} + z_{\bar u,v}\|_{L^2(\Omega)}\|z_{u_\theta,v} - z_{\bar u,v}\|_{L^2(\Omega)}\\
& \le\frac{5}{2}  (C_{L,K_U} + M_UC_{f,K_U})C_{L^2(\Omega)}|\Omega|^{\frac{1}{2}}\|z_{\bar u,v}\|_{L^2(\Omega)}\|y_\theta - \bar y\|_{C(\bar\Omega)}\|z_{\bar u,v}\|_{L^2(\Omega)}\\
&\le \rho\|z_{\bar u,v}\|^2_{L^2(\Omega)} \quad \text{if } \|y_u - \bar y\|_{C(\bar\Omega)} < \varepsilon.
\end{align*}
Hence, \eqref{E3.17} is a straightforward consequence of the above estimates.
\end{proof}

\begin{proof}[Proof of Theorem \ref{T3.3}]
Let us take $u \in \Uad $ with $\|y_u - \bar y\|_{C(\bar\Omega)} < \alpha$. By performing a Taylor expansion and using that $J'(\bar u)(u - \bar u) \ge 0$ we obtain
\begin{align*}
&J(u) = J(\bar u) + J'(\bar u)(u - \bar u) + \frac{1}{2}J''(u_\theta)(u - \bar u)^2\\
& \ge J(\bar u) + \frac{1}{2}[J'(\bar u)(u - \bar u) + J''(\bar u)(u - \bar u)^2] + \frac{1}{2}[J''(u_\theta) - J''(\bar u)](u - \bar u)^2\\
&\ge J(\bar u) + \frac{\gamma}{2}\|z_{\bar u,u - \bar u}\|^2_{L^2(\Omega)} - \frac{1}{2}|[J''(u_\theta) - J''(\bar u)](u - \bar u)^2|.
\end{align*}
Lemma \ref{L3.3} implies the existence of $\varepsilon \in (0,\alpha]$ such that $|[J''(u_\theta) - J''(\bar u)](u - \bar u)^2| < \frac{\gamma}{2}\|z_{\bar u,u - \bar u}\|^2_{L^2(\Omega)}$ for every $u \in \Uad$ with $\|y_u - \bar y\|_{C(\bar\Omega)} < \varepsilon$. Inserting this estimate in the above expression and taking $\varepsilon$ still smaller if necessary, we can apply \eqref{E2.14} to deduce
\[
J(u) \ge J(\bar u) + \frac{\gamma}{4}\|z_{\bar u,u - \bar u}\|^2_{L^2(\Omega)} \ge J(\bar u) + \frac{\gamma}{16}\|y_u - \bar y\|^2_{L^2(\Omega)}.
\]
This inequality yields \eqref{E3.12} with $\kappa = \frac{\gamma}{8}$.
\end{proof}

\section{Stability of the states}
\label{S4}

In this section, we consider the following perturbations of the control problem \Pb
\[
   \Pbe \  \min_{u \in \Uad} J_\varepsilon(u) := \int_\Omega[L(x,y^\varepsilon_u(x),u(x)) + \eta_\varepsilon(x)y^\varepsilon_u(x)]\dx,
\]
where $y_u^\varepsilon$ is the solution of the equation
\begin{align}
\left\{\begin{array}{l}
-\dive\big(A(x)\nabla y\big)+b(x) \cdot \nabla y+f(x,y) = u + \xi_\varepsilon\  \text{ in }\ \Omega,
\\ y=0 \ \text{ on }\ \Gamma.
\end{array} \right.
\label{E4.1}
\end{align}
Here we assume that $\{\xi_{\varepsilon}\}_{\varepsilon > 0}$ and $\{\eta_{\varepsilon}\}_{\varepsilon > 0}$ are bounded families in $L^2(\Omega)$  satisfying that $(\xi_\varepsilon,\eta_\varepsilon)\to (0,0)$ in $L^2(\Omega)^2$ as $\varepsilon \to 0$. As a consequence of Theorem \ref{T2.1} we get the existence and uniqueness of a solution $y_u^\varepsilon \in H_0^1(\Omega) \cap C(\bar\Omega)$ of \eqref{E4.1}. Moreover, using \eqref{E2.7} with $r = 2$ and the boundedness of $\{\xi_\varepsilon\}_{\varepsilon > 0}$ in $L^2(\Omega)$ we infer that the set $\{y_u^\varepsilon : u \in \Uad \text{ and } \varepsilon > 0\}$ is bounded in $H_0^1(\Omega) \cap C(\bar\Omega)$. Therefore, increasing the value of $K_U$, if necessary, we can assume that \eqref{E2.8} and the inequality
\begin{equation}
\|y_u^\varepsilon\|_{H_0^1(\Omega)} + \|y_u^\varepsilon\|_{C(\bar\Omega)} \le K_U\quad \forall u \in \Uad\ \text{ and }\ \forall \varepsilon > 0
\label{E4.2}
\end{equation}
hold. We will prove that the solutions of problems \Pbe converge to the solutions of \Pb in some sense to be made precise below. Conversely, we will also prove that any strict strong local minimizer of \Pb can be approximated by strong local minimizers of problems \Pbe. Finally, the Lipschitz stability of the optimal states with respect to the perturbations is established. We start analyzing the difference between the solutions of \eqref{E1.1} and \eqref{E4.1}.

\begin{theorem}
The following inequalities hold for every $\varepsilon > 0$
\begin{align}
&\|y^\varepsilon_u - y_u\|_{H_0^1(\Omega)} + \|y^\varepsilon_u - y_u\|_{C(\bar\Omega)} \le C_2\|\xi_\varepsilon\|_{L^2(\Omega)} \ \forall u \in L^2(\Omega),\label{E4.3}\\
&\|z^\varepsilon_{u,v} - z_{u,v}\|_{L^2(\Omega)} \le C_2^2C_{f,K_U}\|\xi_\varepsilon\|_{L^2(\Omega)}\|z_{u,v}\|_{L^2(\Omega)}\ \forall (u,v) \in \Uad \times L^2(\Omega),
\label{E4.4}
\end{align}
where $C_2$ is the constant given in \eqref{E2.2} for $r = 2$, $C_{f,K_U}$ is the constant $C_{f,M}$ of \eqref{E2.5} with $M = K_U$ given in \eqref{E2.8} or \eqref{E4.2}, and $z^\varepsilon_{u,v}$ denotes the solution of \eqref{E2.9} with $y^\varepsilon_u$ replacing $y_u$.
\label{T4.1}
\end{theorem}

\begin{proof}
Subtracting the equations \eqref{E4.1} and \eqref{E1.1} and using the mean value theorem we obtain
\[
\mathcal{A}(y^\varepsilon_u - y_u) + \frac{\partial f}{\partial y}(x,y_\theta)(y^\varepsilon_u - y_u) = \xi_\varepsilon.
\]
Then, \eqref{E2.2} implies \eqref{E4.3}. To prove \eqref{E4.4} we subtract the equations satisfied by $z^\varepsilon_{u,v}$ and $z_{u,v}$ to obtain
\[
\mathcal{A}(z^\varepsilon_{u,v} - z_{u,v}) + \frac{\partial f}{\partial y}(x,y^\varepsilon_u)(z^\varepsilon_{u,v} - z_{u,v}) = \Big[\frac{\partial f}{\partial y}(x,y_u) - \frac{\partial f}{\partial y}(x,y^\varepsilon_u)\Big]z_{u,v}.
\]
Now, using \eqref{E2.3} with $s = 2$, \eqref{E2.5}, \eqref{E2.8}, and \eqref{E4.3} we obtain from the previous equation with the mean value theorem
\begin{align*}
&\|z^\varepsilon_{u,v} - z_{u,v}\|_{L^2(\Omega)} \le C_2\Big\|\Big[\frac{\partial f}{\partial y}(x,y_u) - \frac{\partial f}{\partial y}(x,y^\varepsilon_u)\Big]z_{u,v}\Big\|_{L^1(\Omega)}\\
&\le C_2C_{f,K_U}\|(y^\varepsilon_u - y_u)z_{u,v}\|_{L^1(\Omega)}\\
&\le C_2C_{f,K_U}\|y^\varepsilon_u - y_u\|_{L^2(\Omega)}\|z_{u,v}\|_{L^2(\Omega)} \le C_2^2C_{f,K_U}\|\xi_\varepsilon\|_{L^2(\Omega)}\|z_{u,v}\|_{L^2(\Omega)}.
\end{align*}
\hfill\end{proof}

Now we analyze the convergence of problems \Pbe to \Pb.

\begin{theorem}
Let $\{u_\varepsilon\}_{\varepsilon > 0}$ be a family of solutions of problems \Pbe. Any control $\bar u$ that is a weak$^*$ limit in $L^\infty(\Omega)$ of a sequence $\{u_{\varepsilon_k}\}_{k = 1}^\infty$ with $\varepsilon_k \to 0$ as $k \to \infty$ is a solution of \Pb. Moreover, the strong convergence $y_{u_{\varepsilon_k}}^{\varepsilon_k} \to y_{\bar u}$ in $H_0^1(\Omega) \cap C(\bar\Omega)$ holds. 
\label{T4.2}
\end{theorem}

\begin{proof}
The existence of the sequences $\{u_{\varepsilon_k}\}_{k = 1}^\infty$ converging to $\bar u$ weakly$^*$ in $L^\infty(\Omega)$ is a consequence of the boundedness of $\Uad$ in $L^\infty(\Omega)$. From Theorem \ref{T2.1} and \eqref{E4.3} we infer
\begin{align*}
&\|y_{u_{\varepsilon_k}}^{\varepsilon_k} - y_{\bar u}\|_{H_0^1(\Omega)} + \|y_{u_{\varepsilon_k}}^{\varepsilon_k} - y_{\bar u}\|_{C(\bar\Omega)}\\
& \le \|y_{u_{\varepsilon_k}}^{\varepsilon_k} - y_{u_{\varepsilon_k}}\|_{H_0^1(\Omega)} + \|y_{u_{\varepsilon_k}}^{\varepsilon_k} - y_{u_{\varepsilon_k}}\|_{C(\bar\Omega)} +  \|y_{u_{\varepsilon_k}} - y_{\bar u}\|_{H_0^1(\Omega)} + \|y_{u_{\varepsilon_k}} - y_{\bar u}\|_{C(\bar\Omega)}\\
&\le C_2\|\xi_\varepsilon\|_{L^2(\Omega)} + \|y_{u_{\varepsilon_k}} - y_{\bar u}\|_{H_0^1(\Omega)} + \|y_{u_{\varepsilon_k}} - y_{\bar u}\|_{C(\bar\Omega)} \to 0\ \text{ as } k \to \infty.
\end{align*}
Using this fact, the convergence $\eta_\varepsilon \to 0$ as $\varepsilon \to 0$, \eqref{E3.2}, the optimality of $u_{\varepsilon_k}$ for \Pbek, and again \eqref{E4.3}, we get
\[
J(\bar u) = \lim_{k \to \infty}J_{\varepsilon_k}(u_{\varepsilon_k}) \le \lim_{k \to \infty}J_{\varepsilon_k}(u) = J(u)\quad \forall u \in \Uad,
\]
which proves that $\bar u$ is a solution of \Pb.
\end{proof}

Now, we establish a kind of converse result.

\begin{theorem}
Let $\bar u$ be a strict strong local minimizer of \Pb. Then, there exist $\varepsilon_0 > 0$ and a family of strong local minimizers $\{u_\varepsilon\}_{\varepsilon < \varepsilon_0}$ of problems \Pbe such that $u_\varepsilon \stackrel{*}{\rightharpoonup} \bar u$ in $L^\infty(\Omega)$ and $y_{u_\varepsilon}^\varepsilon \to y_{\bar u}$ strongly in $H_0^1(\Omega) \cap C(\bar\Omega)$ as $\varepsilon \to 0$.
\label{T4.3}
\end{theorem}

\begin{proof}
Since $\bar u$ is a strict strong local minimizer of \Pb, there exists $\rho > 0$ such that $\bar u$ is the unique solution of the problem
\[
   \Pbr \  \min_{u \in \Ur} J(u),
\]
where $\Ur = \{u \in \Uad : \|y_u - y_{\bar u}\|_{C(\bar\Omega)} \le \rho\}$. Now, for every $\varepsilon > 0$ we define the problems
\[
   \Pbre \  \min_{u \in \Ur} J_\varepsilon(u).
\]
Using Theorem \ref{T2.1} we deduce that $\Ur$ is weakly$^*$ closed in $L^\infty(\Omega)$, hence the existence of a solution $u_\varepsilon$ of \Pbre can be proved as we indicated for \Pb. Moreover, arguing as in the proof of Theorem \ref{T4.2}, we deduce the existence of sequences $\{u_{\varepsilon_k}\}_{k = 1}^\infty$ converging weakly$^*$ to a solution $u$ of \Pbr in $L^\infty(\Omega)$ and such that  $y^{\varepsilon_k}_{u_{\varepsilon_k}} \to y_u$ strongly in $H_0^1(\Omega) \cap C(\bar\Omega)$. Since $\bar u$ is the unique solution of \Pbr, we conclude the convergence $u_\varepsilon \stackrel{*}{\rightharpoonup} \bar u$ in $L^\infty(\Omega)$ and $y^\varepsilon_{u_\varepsilon} \to y_{\bar u}$ in $H_0^1(\Omega) \cap C(\bar\Omega)$ as $\varepsilon \to 0$. Therefore, there exists $\varepsilon_0 > 0$ such that $\|y^\varepsilon_{u_\varepsilon} - y_{\bar u}\|_{C(\bar\Omega)} < \rho$ for every $\varepsilon < \varepsilon_0$. This implies that $u_\varepsilon$ is a strong local minimizer of \Pbe for every $\varepsilon < \varepsilon_0$, which completes the proof.
\end{proof}

Now we establish our main theorem of this section.

\begin{theorem}
Let $\bar u$ be a local minimizer of \Pb satisfying Assumption \ref{A4} and $\{u_\varepsilon\}_{\varepsilon < \varepsilon_0}$ a family of local solutions of problems \Pbe such that $u_\varepsilon \stackrel{*}{\rightharpoonup} \bar u$ in $L^\infty(\Omega)$ as $\varepsilon \to 0$. Then, there exist $\hat\varepsilon \in (0,\varepsilon_0)$ and  a constant $C > 0$ such that
\begin{equation}
\|y^\varepsilon_{u_\varepsilon} - \bar y\|_{L^2(\Omega)} \le C\Big(\|\xi_\varepsilon\|_{L^2(\Omega)} + \|\eta_\varepsilon\|_{L^2(\Omega)}\Big) \quad \forall \varepsilon < \hat\varepsilon,
\label{E4.5}
\end{equation}
where $\bar y = y_{\bar u}$.
\label{T4.4}
\end{theorem}

Let us observe that Assumption \ref{A4} implies that $\bar u$ satisfies \eqref{E3.12}. Hence, $\bar u$ is a strict strong local minimizer of \Pb and, consequently, Theorem \ref{T4.3} ensures the existence of a family $\{u_\varepsilon\}_{\varepsilon < \varepsilon_0}$ of strong local minimizers of problems \Pbe satisfying the conditions of the above theorem. Before proving this theorem we establish the following lemma.

\begin{lemma}
Let $\bar u$ satisfy the assumptions of Theorem \ref{T4.4}. Then, there exists $\varepsilon > 0$ such that
\begin{equation}
J'(u)(u - \bar u) \ge \frac{\gamma}{2}\|z_{u,u - \bar u}\|^2_{L^2(\Omega)} \quad \forall u \in \Uad \text{ with } \|y_u - \bar y\|_{C(\bar\Omega)} < \varepsilon,
\label{E4.6}
\end{equation}
where $\gamma$ is given in Assumption \ref{A4}.
\label{L4.1}
\end{lemma}

\begin{proof}
We denote by $H:\Omega \times \mathbb{R}^3 \longrightarrow \mathbb{R}$ the Hamiltonian associated with the control problem \Pb:
\[
H(x,y,\varphi,u) = L(x,y,u) + \varphi[u - f(x,y)].
\]
For every $u \in \Uad$ and $v \in L^2(\Omega)$, we define $\psi_{u,v} \in H_0^1(\Omega) \cap C(\bar\Omega)$ as the function satisfying
\[
\mathcal{A}^*\psi_{u,v} + \frac{\partial f}{\partial y}(x,y_u)\psi_{u,v} = \frac{\partial^2H}{\partial y^2}(x,y_u,\varphi_u,u)z_{u,v}.
\]
We split the proof into two steps.

{\em Step I.-} Here we prove that for every $\rho > 0$ there exists $\varepsilon > 0$ such that for every $u \in \Uad$ with $\|y_u - \bar y\|_{C(\bar\Omega)} < \varepsilon$ we have
\begin{equation}
\Big|\int_\Omega(\varphi_u - \bar\varphi - \psi_{\bar u,u - \bar u})(u - \bar u)\dx\Big| \le \rho\|z_{\bar u,u - \bar u}\|^2_{L^2(\Omega)}.
\label{E4.7}
\end{equation}
Setting $\pi = \varphi_u - \bar\varphi - \psi_{\bar u,u - \bar u}$ and subtracting their respective equations it follows with the mean value theorem
\begin{align*}
\mathcal{A}^*\pi + \frac{\partial f}{\partial y}(x,\bar y)\pi &= \frac{\partial H}{\partial y}(x,y_u,\varphi_u,u) - \frac{\partial H}{\partial y}(x,\bar y,\bar\varphi,\bar u)\\
& -\frac{\partial^2H}{\partial y^2}(x,\bar y,\bar\varphi,\bar u)z_{\bar u,u - \bar u} - \frac{\partial^2H}{\partial y\partial\varphi}(x,\bar y,\bar\varphi,\bar u)(\varphi_u - \bar\varphi)\\
&= \frac{\partial^2H}{\partial y^2}(x,y_\theta,\varphi_\theta,u_\theta)(y_u - \bar y) - \frac{\partial^2H}{\partial y^2}(x,\bar y,\bar\varphi,\bar u)z_{\bar u,u - \bar u}\\
& + \Big[\frac{\partial^2H}{\partial y\partial\varphi}(x,y_\theta,\varphi_\theta,u_\theta) - \frac{\partial^2H}{\partial y\partial\varphi}(x,\bar y,\bar\varphi,\bar u)\Big](\varphi_u - \bar\varphi)\\
&=\frac{\partial^2H}{\partial y^2}(x,y_\theta,\varphi_\theta,u_\theta)(y_u - \bar y - z_{\bar u,u - \bar u})\\
& + \Big[\frac{\partial^2H}{\partial y^2}(x,y_\theta,\varphi_\theta,u_\theta) - \frac{\partial^2H}{\partial y^2}(x,\bar y,\bar\varphi,\bar u)\Big]z_{\bar u,u - \bar u}\\
&+ \Big[\frac{\partial^2H}{\partial y\partial\varphi}(x,y_\theta,\varphi_\theta,u_\theta) - \frac{\partial^2H}{\partial y\partial\varphi}(x,\bar y,\bar\varphi,\bar u)\Big](\varphi_u - \bar\varphi).
\end{align*}
This implies
\begin{align*}
\int_\Omega\pi&(u - \bar u)\dx = \int_\Omega\pi\Big(\mathcal{A}z_{\bar u,u - \bar u} + \frac{\partial f}{\partial y}(x,\bar y)z_{\bar u,u - \bar u}\Big)\dx\\
& = \int_\Omega\Big(\mathcal{A}^*\pi + \frac{\partial f}{\partial y}(x,\bar y)\pi\Big)z_{\bar u,u - \bar u}\dx\\
& = \int_\Omega\frac{\partial^2H}{\partial y^2}(x,y_\theta,\varphi_\theta,u_\theta)(y_u - \bar y - z_{\bar u,u - \bar u})z_{\bar u,u - \bar u}\dx\\
& + \int_\Omega\Big[\frac{\partial^2H}{\partial y^2}(x,y_\theta,\varphi_\theta,u_\theta) - \frac{\partial^2H}{\partial y^2}(x,\bar y,\bar\varphi,\bar u)\Big]z^2_{\bar u,u - \bar u}\dx\\
& + \int_{\Omega}\Big[\frac{\partial^2H}{\partial y\partial\varphi}(x,y_\theta,\varphi_\theta,u_\theta) - \frac{\partial^2H}{\partial y\partial\varphi}(x,\bar y,\bar\varphi,\bar u)\Big](\varphi_u - \bar\varphi)z_{\bar u,u - \bar u}\dx = I_1+I_2+I_3.
\end{align*}
We estimate every term $I_i$. For the first term we use \eqref{E2.5}, \eqref{E2.8}, \eqref{E2.12} with $s = 2$, \eqref{E2.14} with $X = L^2(\Omega)$, \eqref{E3.2},  and \eqref{E3.15} as follows
\begin{align*}
|I_1|& \le (C_{L,K_U} + M_UC_{f,K_U})\|y_u - \bar y - z_{\bar u,u - \bar u}\|_{L^2(\Omega)}\|z_{\bar u,u - \bar u}\|_{L^2(\Omega)}\\
& \le (C_{L,K_U} + M_UC_{f,K_U})M_2\|y_u - \bar y\|^2_{L^2(\Omega)}\|z_{\bar u,u - \bar u}\|_{L^2(\Omega)}\\
& \le 2(C_{L,K_U} + M_UC_{f,K_U})M_2\sqrt{|\Omega|}\varepsilon\|z_{\bar u,u - \bar u}\|^2_{L^2(\Omega)}.
\end{align*}
The second term is estimated with \eqref{E2.6}, \eqref{E2.8}, \eqref{E3.3}, \eqref{E3.14}, \eqref{E3.15}, \eqref{E3.16}, leading to $|I_2| \le \rho\|z_{\bar u,u - \bar u}\|^2_{L^2(\Omega)}$ for $\rho$ arbitrarily small if $\varepsilon$ is taken according to $\rho$. Finally, for the last term we use the same inequalities as for $I_2$, that estimate \eqref{E3.16} holds true for $L^2(\Omega)$ instead of $C(\bar \Omega)$ and additionally \eqref{E2.15} with $X = L^2(\Omega)$ to get
\begin{align*}
|I_3| &\le \rho\|\varphi_u - \bar\varphi\|_{L^2(\Omega)}\|z_{\bar u,u - \bar u}\|_{L^2(\Omega)}\\
& \le \rho C_2(C_{L,K_U} + M_UC_{f,K_U})\|y_u - \bar y\|_{L^2(\Omega)}\|z_{\bar u,u - \bar u}\|_{L^2(\Omega)}\\
& \le 2\rho C_2(C_{L,K_U} + M_UC_{f,K_U})\|z_{\bar u,u - \bar u}\|^2_{L^2(\Omega)},
\end{align*}
where again $\rho$ is arbitrarily small if $\varepsilon$ is chosen according to it. Thus, \eqref{E4.7} follows from the proved estimates.\vspace{2mm}

{\em Step II-} Now, we prove \eqref{E4.6}. First, we observe that for every $v \in L^2(\Omega)$
\begin{align*}
&\int_\Omega\psi_{\bar u,v}v\dx = \int_\Omega\psi_{\bar u,v}\Big(\mathcal{A}z_{\bar u,v} + \frac{\partial f}{\partial y}(x,\bar y)z_{\bar u,v}\Big)\dx\\
& = \int_\Omega\Big(\mathcal{A}^*\psi_{\bar u,v} + \frac{\partial f}{\partial y}(x,\bar y)\psi_{\bar u,v}\Big)z_{\bar u,v}\dx = \int_\Omega\frac{\partial^2H}{\partial y^2}(x,\bar y,\bar\varphi,\bar u)z_{\bar u,v}^2\dx = J''(\bar u)v^2,
\end{align*}
where the last inequality follows from \eqref{E3.5} and the definition of the Hamiltonian. Let $\varepsilon > 0$ be such that \eqref{E4.7} holds with  $\rho = \frac{\gamma}{2}$. Then, using Assumption \ref{A4} and \eqref{E4.7} we get for $u \in \Uad$ with $\|y_u - \bar y\|_{C(\bar\Omega)} < \varepsilon$

\begin{align*}
&J'(u)(u - \bar u) = \int_\Omega(\varphi_u + g)(u - \bar u)\dx\\
& = \int_\Omega(\varphi_u - \bar\varphi - \psi_{\bar u,u - \bar u})(u - \bar u)\dx + \int_\Omega(\bar\varphi + g + \psi_{\bar u,u - \bar u})(u - \bar u)\dx\\
&\ge -\frac{\gamma}{2}\|z_{\bar u,u - \bar u}\|^2_{L^2(\Omega)} + [J'(\bar u)(u - \bar u) + J''(\bar u)(u - \bar u)^2] \ge \frac{\gamma}{2}\|z_{\bar u,u - \bar u}\|^2_{L^2(\Omega)}.
\end{align*}
\end{proof}

\begin{remark}
Let us notice that if $\bar u$ is a local minimizer of \Pb satisfying Assumption \ref{A4}, then there exists $\varepsilon > 0$ such that there is no stationary point $\hat u$ of \Pb different from $\bar u$ such that $\|y_{\hat u} - \bar y\|_{C(\bar\Omega)} < \varepsilon$. We say that $\hat u$ is a stationary point of \Pb if it satisfies the first order optimality condition. In particular, if $\hat u$ is a stationary point then $J'(\hat u)(\bar u - \hat u) \ge 0$. This contradicts \eqref{E4.6} if $\|y_{\hat u} - \bar y\|_{C(\bar\Omega)} < \varepsilon$.
\label{R4.1}
\end{remark}\vspace{2mm}

\begin{proof}[Proof of Theorem \ref{T4.4}]
Using the local optimality of $u_\varepsilon$ we get
\begin{align}
0 &\ge J'_\varepsilon(u_\varepsilon)(u_\varepsilon - \bar u)\notag\\
& = J'(u_\varepsilon)(u_\varepsilon - \bar u) + \int_\Omega\Big[\frac{\partial L}{\partial y}(x,y^\varepsilon_{u_\varepsilon},u_\varepsilon) - \frac{\partial L}{\partial y}(x,y_{u_\varepsilon},u_\varepsilon)\Big]z_{u_\varepsilon,u_\varepsilon - \bar u}\dx\notag\\
&+ \int_\Omega\frac{\partial L}{\partial y}(x,y^\varepsilon_{u_\varepsilon},u_\varepsilon) (z^\varepsilon_{u_\varepsilon,u_\varepsilon - \bar u} - z_{u_\varepsilon,u_\varepsilon - \bar u})\dx + \int_\Omega\eta_\varepsilon z^\varepsilon_{u_\varepsilon,u_\varepsilon - \bar u}\dx. \label{E4.8}
\end{align}
We estimate each one of these four terms. First, we observe that the convergence $u_\varepsilon \rightharpoonup \bar u$ in $L^2(\Omega)$ implies that $\|y_{u_\varepsilon} - \bar y\|_{C(\bar\Omega)} \to 0$; see Theorem \ref{T2.1}. Hence, from Lemma \ref{L4.1} we deduce the existence of $\varepsilon_1 > 0$ such that
\begin{equation}
J'(u_\varepsilon)(u_\varepsilon - \bar u) \ge \frac{\gamma}{2}\|z_{u_\varepsilon,u_\varepsilon - \bar u}\|^2_{L^2(\Omega)}\quad \forall \varepsilon < \varepsilon_1.
\label{E4.9}
\end{equation}
For the second term we use Schwarz's inequality, the mean value theorem, \eqref{E2.8} and \eqref{E4.2}, \eqref{E3.2}, and \eqref{E4.3}
\begin{align}
&\int_\Omega\Big|\frac{\partial L}{\partial y}(x,y^\varepsilon_{u_\varepsilon},u_\varepsilon) - \frac{\partial L}{\partial y}(x,y_{u_\varepsilon},u_\varepsilon)\Big| |z_{u_\varepsilon,u_\varepsilon - \bar u}|\dx\notag\\
& \le C_{L,K_U}\|y^\varepsilon_{u_\varepsilon} - y_{u_\varepsilon}\|_{L^2(\Omega)}\|z_{u_\varepsilon,u_\varepsilon - \bar u}\|_{L^2(\Omega)}\notag\\
& \le C_{L,K_U}\sqrt{|\Omega|}C_2\|\xi_\varepsilon\|_{L^2(\Omega)}\|z_{u_\varepsilon,u_\varepsilon - \bar u}\|_{L^2(\Omega)}.
\label{E4.10}
\end{align}
Now we estimate the third term with \eqref{E3.2} and \eqref{E4.2}, Schwarz's inequality, and \eqref{E4.4}
\begin{align}
\int_\Omega\Big|\frac{\partial L}{\partial y}(x,&y^\varepsilon_{u_\varepsilon},u_\varepsilon)\Big| |z^\varepsilon_{u_\varepsilon,u_\varepsilon - \bar u} - z_{u_\varepsilon,u_\varepsilon - \bar u}|\dx \le \int_\Omega\psi_{K_U}|z^\varepsilon_{u_\varepsilon,u_\varepsilon - \bar u} - z_{u_\varepsilon,u_\varepsilon - \bar u}|\dx\notag\\
&\leq\|\psi_{K_U}\|_{L^2(\Omega)}C_2^2C_{f,K_U}\|\xi_\varepsilon\|_{L^2(\Omega)}\|z_{u_\varepsilon,u_\varepsilon - \bar u}\|_{L^2(\Omega)}.
\label{E4.11}
\end{align}
For the last term we use again \eqref{E4.4} and the fact that $\{\xi_\varepsilon\}_{\varepsilon > 0}$ is bounded in $L^2(\Omega)$
\begin{align}
&\int_\Omega|\eta_\varepsilon z^\varepsilon_{u_\varepsilon,u_\varepsilon - \bar u}|\dx \le \|\eta_\varepsilon\|_{L^2(\Omega)}\Big(\|z^\varepsilon_{u_\varepsilon,u_\varepsilon - \bar u} - z_{u_\varepsilon,u_\varepsilon - \bar u}\|_{L^2(\Omega)} + \|z_{u_\varepsilon,u_\varepsilon - \bar u}\|_{L^2(\Omega)}\Big)\notag\\
&\le \Big(C_2^2C_{f,K_U}\|\xi_\varepsilon\|_{L^2(\Omega)} + 1\Big)\|\eta_\varepsilon\|_{L^2(\Omega)}\|z_{u_\varepsilon,u_\varepsilon - \bar u}\|_{L^2(\Omega)} \le C\|\eta_\varepsilon\|_{L^2(\Omega)}\| z_{u_\varepsilon,u_\varepsilon - \bar u}\|_{L^2(\Omega)}.
\label{E4.12}
\end{align}
Inserting the estimates \eqref{E4.9}--\eqref{E4.12} in \eqref{E4.8} we obtain for some constant $C' > 0$ and every $\varepsilon < \varepsilon_1$
\[
\|z_{u_\varepsilon,u_\varepsilon - \bar u}\|_{L^2(\Omega)} \le C'\Big(\|\xi_\varepsilon\|_{L^2(\Omega)} + \|\eta_\varepsilon\|_{L^2(\Omega)}\Big).
\]
Finally, using \eqref{E2.14} and \eqref{E4.3} we deduce the existence of $\varepsilon_2 \in (0,\varepsilon_1]$ such that for every $\varepsilon < \varepsilon_2$ we have
\begin{align*}
&\|y^\varepsilon_{u_\varepsilon} - \bar y\|_{L^2(\Omega)} \le \|y^\varepsilon_{u_\varepsilon} - y_{u_\varepsilon}\|_{L^2(\Omega)} + \|y_{u_\varepsilon} - \bar y\|_{L^2(\Omega)}\\
&\le C_2\sqrt{|\Omega|}\|\xi_\varepsilon\|_{L^2(\Omega)} + 2\|z_{u_\varepsilon,u_\varepsilon - \bar u}\|_{L^2(\Omega)}\\ &\le C_2\sqrt{|\Omega|} \|\xi_\varepsilon\|_{L^2(\Omega)} + 2C'\Big(\|\xi_\varepsilon\|_{L^2(\Omega)} + \|\eta_\varepsilon\|_{L^2(\Omega)}\Big),
\end{align*}
which proves \eqref{E4.5}.
\end{proof}

\section{Stability of the controls}
\label{S5}

In the previous section, we established Lipschitz stability for the optimal states with respect to state perturbations in the objective functional and to the force in the state equation. In order to obtain stability of the optimal controls an additional assumption is usually required. The reader is referred to
\cite{Qui-Wachsmuth2018} for the following assumption
\begin{equation}
\exists C > 0 \text{ such that } |\{x \in \Omega : |(\varphi + g)(x)| \le \varepsilon\}| \le C\varepsilon\quad \forall \varepsilon > 0.
\label{E5.1}
\end{equation}
Using this assumption and sufficient second order optimality conditions they proved Lipschitz stability of the controls in the $L^1(\Omega)$ norm. However, the assumption \eqref{E5.1} implies that $\bar u$ is bang-bang. As far as we know, there is no proof for stability of the optimal controls when they are not bang-bang. Assumption \ref{A4} that we have considered in the previous sections is applicable for the case of optimal controls that are not bang-bang. Nevertheless, it leads only to Lipschitz stability of the optimal states. Here, we modify Assumption \ref{A4} as follows
\begin{assumption}
There exist numbers $\alpha > 0$ and $\gamma > 0$ such that for all $u \in \Uad$ with $\|y_u - \bar y\|_{C(\bar\Omega)} < \alpha$ the following inequality is fulfilled
\begin{equation}
J'(\bar u)(u - \bar u) + J''(\bar u)(u - \bar u)^2 \ge \gamma\|z_{\bar u,u - \bar u}\|_{L^2(\Omega)}\|u - \bar u\|_{L^1(\Omega)}.
\label{E5.2}
\end{equation}
\label{A5}
\end{assumption}
Under this assumption we will prove Lipschitz stability of the optimal controls. 
It has been proved in \cite{DJV2022} that the sufficient second order conditions plus the structural assumption \eqref{E5.1} imply the existence of positive numbers $\gamma$ and $\alpha$ such that
\begin{equation}
J'(\bar u)(u - \bar u) + J''(\bar u)(u - \bar u)^2 \ge \gamma\|u - \bar u\|^2_{L^1(\Omega)}\ \forall u \in \Uad \text{ with } \|u - \bar u\|_{L^1(\Omega)} < \alpha.
\label{E5.3}
\end{equation}
But we have the next equivalence:

\begin{proposition}
The statement \eqref{E5.3} is equivalent to the existence of positive numbers $\gamma'$ and $\alpha'$ such that
\begin{equation} 
J'(\bar u)(u - \bar u) + J''(\bar u)(u - \bar u)^2 \ge \gamma'\|u - \bar u\|^2_{L^1(\Omega)}\ \forall u \in \Uad \text{ with } \|y_u - \bar y\|_{C(\bar\Omega)} < \alpha'.
\label{E5.4}
\end{equation}
\end{proposition}

\begin{proof}
Let us assume that \eqref{E5.3} holds, but \eqref{E5.4} is false. Then, for every integer $k \ge 1$ there exists an element $u_k \in \Uad$ such that
\begin{equation}
J'(\bar u)(u_k - \bar u) + J''(\bar u)(u_k - \bar u)^2 < \frac{1}{k}\|u_k - \bar u\|^2_{L^1(\Omega)}\  \text{ and } \ \|y_{u_k} - \bar y\|_{C(\bar\Omega)} < \frac{1}{k}.
\label{E5.5}
\end{equation}
Since $\{u_k\}_{k = 1}^\infty \subset \Uad$ is bounded in $L^\infty(\Omega)$, we can extract a subsequence, denoted in the same way, such that $u_k \stackrel{*}{\rightharpoonup} u$ in $L^\infty(\Omega)$. On one side, \eqref{E5.5} implies that $y_{u_k} \to \bar y$ in $C(\bar\Omega)$. On the other side, from Theorem \ref{T2.1} the convergence $y_{u_k} \to y_u$ in $C(\bar\Omega)$ follows. Then, $y_u = \bar y$ and, consequently, $u = \bar u$ holds.  But \eqref{E5.3} implies that $\bar u$ is bang-bang and, hence, the weak convergence $u_k \stackrel{*}{\rightharpoonup} \bar u$ yields the strong convergence $u_k \to \bar u$ in $L^1(\Omega)$; see \cite[Proposition 12 and Lemma 6]{DJV2022}. Then, \eqref{E5.5} contradicts \eqref{E5.3}.

Let us prove the converse implication. First we observe that given $u \in \Uad$ we get with the mean value theorem
\[
\mathcal{A}(y_u - \bar y) + \frac{\partial f}{\partial y}(x,\bar y + \theta(y_u - \bar y))(y_u - \bar y) = u - \bar u.
\]
Now, using \eqref{E2.2} with $r = 2$ we get
\[
\|y_u - \bar y\|_{C(\bar\Omega)} \le C_2\|u - \bar u\|_{L^2(\Omega)} \le C_2\sqrt{u_b - u_a}\|u - \bar u\|^{\frac{1}{2}}_{L^1(\Omega)}.
\]
Then, taking $\alpha = \frac{\alpha'^2}{C^2_2(u_b - u_a)}$, we obtain that \eqref{E5.4} implies \eqref{E5.3} with $\gamma = \gamma'$.
\end{proof}

From \eqref{E2.3} we infer that \eqref{E5.4} implies \eqref{E5.2}. Hence, the combination of sufficient second order conditions plus \eqref{E5.1} is a stronger assumption than \eqref{E5.2}.

\begin{theorem}
Let $\bar u$ be a local minimizer of \Pb satisfying Assumption \ref{A5} and $\{u_\varepsilon\}_{\varepsilon < \varepsilon_0}$ a family of local solutions of problems \Pbe such that $u_\varepsilon \stackrel{*}{\rightharpoonup} \bar u$ in $L^\infty(\Omega)$ as $\varepsilon \to 0$. Then, there exist $\hat\varepsilon \in (0,\varepsilon_0)$ and  a constant $C > 0$ such that
\begin{equation}
\|u_{\varepsilon} - \bar u\|_{L^1(\Omega)} \le C\Big(\|\xi_\varepsilon\|_{L^2(\Omega)} + \|\eta_\varepsilon\|_{L^2(\Omega)}\Big) \quad \forall \varepsilon < \hat\varepsilon,
\label{E5.6}
\end{equation}
where $\bar y = y_{\bar u}$.
\label{T5.1}
\end{theorem}

The proof of this theorem follows the steps of the one of Theorem \ref{T4.4} with Lemma \ref{L4.1} replaced by the following:

\begin{lemma}
Let $\bar u$ satisfy the assumptions of Theorem \ref{T5.1}. Then, there exists $\varepsilon > 0$ such that
\begin{equation}
J'(u)(u - \bar u) \ge \frac{\gamma}{2}\|z_{u,u - \bar u}\|_{L^2(\Omega)}\|u - \bar u\|_{L^1(\Omega)}\ \forall u \in \Uad \text{ with } \|y_u - \bar y\|_{C(\bar\Omega)} < \varepsilon,
\label{E5.7}
\end{equation}
where $\gamma$ is given in Assumption \ref{A5}.
\label{L5.1}
\end{lemma}

\begin{proof}
We use \eqref{E4.7} with $\rho = \frac{\gamma}{2C_2}$, Assumption \ref{A5}, and \eqref{E2.3} to deduce for $\varepsilon > 0$ small enough
\begin{align*}
&J'(u)(u - \bar u) = \int_\Omega(\varphi_u + g)(u - \bar u)\dx\\
& = \int_\Omega(\varphi_u - \bar\varphi - \psi_{\bar u,u - \bar u})(u - \bar u)\dx + \int_\Omega(\bar\varphi + g + \psi_{\bar u,u - \bar u})(u - \bar u)\dx\\
&\ge -\frac{\gamma}{2C_2}\|z_{\bar u,u - \bar u}\|^2_{L^2(\Omega)} + [J'(\bar u)(u - \bar u) + J''(\bar u)(u - \bar u)^2]\\
& \ge -\frac{\gamma}{2}\|z_{\bar u,u - \bar u}\|_{L^2(\Omega)}\|u - \bar u\|_{L^1(\Omega)} + \gamma\|z_{\bar u,u - \bar u}\|_{L^2(\Omega)}\|u - \bar u\|_{L^1(\Omega)},
\end{align*}
which proves \eqref{E5.7}.
\end{proof}

\begin{proof}[Proof of Theorem \ref{T5.1}] We follow the proof of Theorem \ref{T4.4} replacing the estimate \eqref{E4.9} by \eqref{E5.7} to deduce with \eqref{E4.8} and \eqref{E4.10}--\eqref{E4.12} the inequality
\begin{align*}
0 \ge J'_\varepsilon(u_\varepsilon)(u_\varepsilon - \bar u) &\ge \frac{\gamma}{2}\|z_{u_\varepsilon,u_\varepsilon - \bar u}\|_{L^2(\Omega)}\|u_\varepsilon - \bar u\|_{L^1(\Omega)}\\
& - C_1\|z_{u_\varepsilon,u_\varepsilon - \bar u}\|_{L^2(\Omega)}\Big(\|\xi_\varepsilon\|_{L^2(\Omega)} + \|\eta_\varepsilon\|_{L^2(\Omega)}\Big).
\end{align*}
Then, dividing this inequality by $\|z_{u_\varepsilon,u_\varepsilon - \bar u}\|_{L^2(\Omega)}$ we get
\[
\|u_\varepsilon - \bar u\|_{L^1(\Omega)} \le \frac{2C_1}{\gamma}\Big(\|\xi_\varepsilon\|_{L^2(\Omega)} + \|\eta_\varepsilon\|_{L^2(\Omega)}\Big),
\]
which proves \eqref{E5.6} with $C = \frac{2C_1}{\gamma}$. 
\end{proof}

\section{Some final state stability results}
\label{S6}

In this section we see how Assumption \ref{A5} allows us to prove Lipschitz stability for the optimal states for more general perturbations of \Pb. Here, we consider simultaneous perturbations on the control and state variables of \Pb:
\[
\Pbe \  \min_{u \in \Uad} J_\epsilon(u) := \int_\Omega L_\varepsilon(x,y^\varepsilon_{u}(x),u(x))\dx,
\]
where $y^\varepsilon_u$ is the solution of \eqref{E4.1} and for every $\epsilon > 0$
\[
L_\varepsilon(x,y,u) = L_0(x,y) + \eta_\varepsilon y + g_\varepsilon u + \frac{\varepsilon}{2}u^2.
\]

As in Section \ref{S4}, we assume that $\{\xi_{\varepsilon}\}_{\varepsilon > 0}$ and $\{\eta_{\varepsilon}\}_{\varepsilon > 0}$ are bounded families in $L^2(\Omega)$  satisfying that $(\xi_\varepsilon,\eta_\varepsilon)\to (0,0)$ in $L^2(\Omega)^2$ as $\varepsilon \to 0$. Moreover, we suppose that $\|g_\varepsilon - g\|_{L^\infty(\Omega)} \to 0$ as $\varepsilon \to 0$. Under these assumptions, it is immediate to check that \Pbe is an approximation of \Pb in the sense of Theorems \ref{T4.2} and \ref{T4.3}. Moreover, we have the following Lipschitz stability property for the optimal states:

\begin{theorem}
Let $\bar u$ be a local minimizer of \Pb satisfying Assumption \ref{A5} and $\{u_\varepsilon\}_{\varepsilon < \varepsilon_0}$ a family of local solutions of problems \Pbe such that $u_\varepsilon \stackrel{*}{\rightharpoonup} \bar u$ in $L^\infty(\Omega)$ as $\varepsilon \to 0$. Then, there exist $\hat\varepsilon \in (0,\varepsilon_0)$ and  a constant $C > 0$ such that
\begin{equation}
\|y^\varepsilon_{u_\varepsilon} - \bar y\|_{L^2(\Omega)} \le C\Big(\|\xi_\varepsilon\|_{L^2(\Omega)} + \|\eta_\varepsilon\|_{L^2(\Omega)} + \|g_\varepsilon - g\|_{L^\infty(\Omega)} + \varepsilon\Big) \quad \forall \varepsilon < \hat\varepsilon,
\label{E6.1}
\end{equation}
where $\bar y = y_{\bar u}$.
\label{T6.1}
\end{theorem}

\begin{proof}
Similarly to \eqref{E4.8} we have
\begin{align*}
0 &\ge J'_\varepsilon(u_\varepsilon)(u_\varepsilon - \bar u)  = J'(u_\varepsilon)(u_\varepsilon - \bar u) + \int_\Omega(\varepsilon u_\varepsilon + g_\varepsilon - g)(u_\varepsilon - \bar u)\dx\\
& + \int_\Omega\Big[\frac{\partial L}{\partial y}(x,y^\varepsilon_{u_\varepsilon},u_\varepsilon) - \frac{\partial L}{\partial y}(x,y_{u_\varepsilon},u_\varepsilon)\Big]z_{u_\varepsilon,u_\varepsilon - \bar u}\dx\\
&+ \int_\Omega\frac{\partial L}{\partial y}(x,y^\varepsilon_{u_\varepsilon},u_\varepsilon) (z^\varepsilon_{u_\varepsilon,u_\varepsilon - \bar u} - z_{u_\varepsilon,u_\varepsilon - \bar u})\dx + \int_\Omega\eta_\varepsilon z^\varepsilon_{u_\varepsilon,u_\varepsilon - \bar u}\dx.
\end{align*}
Then, using \eqref{E5.7} and \eqref{E4.10}--\eqref{E4.12} we obtain with \eqref{E2.3}
\begin{align*}
0 &\ge \frac{\gamma}{2}\|z_{u_\varepsilon,u_\varepsilon - \bar u}\|_{L^2(\Omega)}\|u_\varepsilon - \bar u\|_{L^1(\Omega)} - \Big(\varepsilon\|u_\varepsilon\|_{L^\infty(\Omega)} + \|g_\varepsilon - g\|_{L^\infty(\Omega)}\Big)\|u_\varepsilon - \bar u\|_{L^1(\Omega)}\\
& - C_1\|z_{u_\varepsilon,u_\varepsilon - \bar u}\|_{L^2(\Omega)}\Big(\|\xi_\varepsilon\|_{L^2(\Omega)} + \|\eta_\varepsilon\|_{L^2(\Omega)}\Big) \ge \frac{\gamma}{2}\|z_{u_\varepsilon,u_\varepsilon - \bar u}\|_{L^2(\Omega)}\|u_\varepsilon - \bar u\|_{L^1(\Omega)}\\
& - C'\Big(\varepsilon + \|g_\varepsilon - g\|_{L^\infty(\Omega)} + \|\xi_\varepsilon\|_{L^2(\Omega)} + \|\eta_\varepsilon\|_{L^2(\Omega)}\Big)\|u_\varepsilon - \bar u\|_{L^1(\Omega)},
\end{align*}
where $C' = \max\{1,|u_a|,|u_b|,C_1C_2\}$. Dividing the above expression by $\|u_\varepsilon - \bar u\|_{L^1(\Omega)}$ and using \eqref{E2.14} we infer
\[
\|y_{u_\varepsilon} - \bar y\|_{L^2(\Omega)} \le \frac{4C'}{\gamma}\Big(\varepsilon + \|g_\varepsilon - g\|_{L^\infty(\Omega)} + \|\xi_\varepsilon\|_{L^2(\Omega)} + \|\eta_\varepsilon\|_{L^2(\Omega)}\Big).
\]
Now, the rest follows as in the proof of Theorem \ref{T4.4}.
\end{proof}

\bibliographystyle{plain}

\end{document}